\documentclass[12pt,english]{article}
\usepackage[T1]{fontenc}
\usepackage[latin9]{inputenc}
\usepackage{prettyref}
\usepackage{mathtools}
\usepackage{amsmath}
\usepackage{amsthm}
\usepackage[dot]{bibtopic}

\makeatletter
\theoremstyle{plain}
\newtheorem{thm}{\protect\theoremname}[section]
  \theoremstyle{plain}
  \newtheorem{prop}[thm]{\protect\propositionname}
  \theoremstyle{plain}
  \newtheorem{lem}[thm]{\protect\lemmaname}
  \theoremstyle{plain}
  \newtheorem{cor}[thm]{\protect\corollaryname}
  \theoremstyle{remark}

\newrefformat{thm}{\theoremname\space\ref{#1}}
\newrefformat{lem}{\lemmaname\space\ref{#1}}
\newrefformat{prop}{\propositionname\space\ref{#1}}
\newrefformat{cor}{\corollaryname\space\ref{#1}}

\let\C\undefined
\usepackage{constants}

\makeatother

\usepackage{babel}
  \providecommand{\corollaryname}{Corollary}
  \providecommand{\lemmaname}{Lemma}
  \providecommand{\propositionname}{Proposition}
  \providecommand{\remarkname}{Remark}
\providecommand{\theoremname}{Theorem}

\begin{document}
\global\long\def\meas{\operatorname{meas}}
\global\long\def\E{\mathbb{\mathbf{E}}}
 \global\long\def\Re{\operatorname{Re}}
\global\long\def\Im{\operatorname{Im}}
\global\long\def\sgn{\operatorname{sgn}}

\title{The $a$-values of the Riemann zeta function near the critical line}

\author{Junsoo Ha and Yoonbok Lee\footnote{Research of the second author was supported by the Incheon National University Research Grant in 2016.}}
\maketitle

\abstract{ We study the value distribution of the Riemann zeta function near the line $\Re s = 1/2$. We find an asymptotic formula for the number of $a$-values in the rectangle $ 1/2 + h_1 / (\log T)^\theta \leq \Re s \leq 1/2+ h_2 /(\log T)^\theta $, $T \leq \Im s \leq 2T$ for fixed $h_1, h_2>0$ and $ 0 < \theta <1/13$. To prove it, we need an extension of the valid range of Lamzouri, Lester and Radziwi\l\l's recent results on the discrepancy between the distribution of $\zeta(s)$ and its random model. We also propose the secondary main term for the Selberg's central limit theorem by providing sharper estimates on the line $\Re s = 1/2 + 1/(\log T)^\theta $. }

\section{Introduction}

The study of the Riemann zeta function $\zeta(s)$ has been one of the central topics
in analytic number theory. Along with its importance in prime number
theorem, various aspects of this function has been studied. Some of these studies involve the general value distribution of zeta functions in the critical strip, which has its own interests. 

Let $a$ be a nonzero complex number. The solutions to $\zeta(s)=a$, which we denote by $ \rho_a = \beta_a + i \gamma_a$, are called the $a$-values of $\zeta(s)$. We first introduce the some basic facts about them.  By a general theorem of Dirichlet series, there is a $A>0$ depending $a$ such that there are no $a$-values of $ \zeta(s)$ on $ \Re s \geq A$. 
There is a number $ N_0 (a) >0$ such that there is an $a$-value of $\zeta(s)$ very close to $s=-2n$ for each $n \geq N_0 (a)$ and there are at most finitely many other $a$-values in $\Re s\leq 0 $. (See \cite{GLM} or \cite{L}.) Thus, the remaining $a$-values lie in a strip $ 0 < \Re s < A$. For these $a$-values we have
$$ N_a (T) := \sum_{\substack{  \beta_a > 0  \\ 0<\gamma_a <T}} 1 = \frac{T}{2 \pi} \log \frac{T}{2 \pi e}   + O_a ( \log T) $$
for $ a \neq 1 $ and
$$ N_a (T) =   \frac{T}{2 \pi} \log \frac{T}{4 \pi  e }   + O_a ( \log T) $$
for $ a=1$. 

Selberg observed various aspects of $a$-values of $\zeta(s)$. For example, he showed that at least $1/2$ of the nontrivial $a$-values of $\zeta(s)$ lie to the left of $\Re s = 1/2$ assuming the Riemann hypothesis. He also conjectured that approximately $3/4$ of the nontrivial $a$-values lie on the strip $ 0 < \Re s < 1/2$. He did not publish these ideas, however Tsang wrote them with proofs in his thesis \cite[Chapter 1]{Ts}. We refer \cite{Sel} for the extension of this idea to a linear combination of $L$-functions and  \cite{GLM} for a simplicity of $a$-values. 

We now focus on the $a$-values to the right of $1/2$-line. For fixed $ 1/2 < \sigma_1 < \sigma_2 $,  Borchsenius and Jessen \cite{BJ} proved that there exists a constant $ c(a, \sigma_1, \sigma_2 ) $ such that
\begin{equation}\label{eqn:1}
  N_a ( \sigma_1, \sigma_2 ; T) := \sum_{\substack{  \sigma_1 <  \beta_a < \sigma_2   \\ 0<\gamma_a <T}} 1  \sim c(a, \sigma_1, \sigma_2 )T 
 \end{equation}
as $ T \to \infty $, where 
$$ c(a, \sigma_1 , \sigma_2 ) = \int_{\sigma_1}^{\sigma_2} g_a(u) du$$
 for some real valued function $g_a$ with $ g_a (u) >0 $ for $ 1/2 < u \leq 1 $. Recently, Lamzouri, Lester and Radziwi\l\l\  \cite{LLR} reduced the size of the error in \eqref{eqn:1} and proved that
 $$ N_a ( \sigma_1, \sigma_2 ; T) = c(a, \sigma_1, \sigma_2 )T + O \bigg(   T  \frac{ \log \log T}{   ( \log T)^{ \sigma_1/2}}\bigg) $$  
 holds for fixed $ 1/2 < \sigma_1 < \sigma_2 < 1 $  and $ T \geq 3$.

In this paper, we prove analog of the above estimates near the $1/2$-line. 
\begin{thm} \label{thm:aval}
 Let $0<\theta<1/13$ be fixed, $T \geq 3$ and 
 $$ \sigma_T := \sigma_T ( \theta) = \frac12+ \frac{1}{(\log T)^\theta}. $$ 
 Then
\begin{align*}
N_{a}(\sigma_{T} ;T,2T) := \sum_{\substack{   \sigma_T < \beta_a   \\ T <\gamma_a <2 T}} 1 = \frac{T(\log T)^{\theta}}{8\pi^{3/2}\sqrt{\theta}\sqrt{\log\log T}}+O\bigg(T\frac{(\log T)^{\theta}}{(\log\log T)^{3/4}}\bigg).
\end{align*}
\end{thm}

\begin{cor}\label{cor:aval}
Let $ 0 < \theta < 1/13 $, $ 0< h_1 < h_2 $ be fixed and $T \geq 3$. Define $ \sigma_ i = 1/2 + h_i / ( \log T)^\theta$ for $ i=1,2$. Then
\[
N_{a}(\sigma_{1},\sigma_{2}; T)= \frac{(h_1^{-1} - h_2^{-1})T(\log T)^{\theta}}{8\pi^{3/2}\sqrt{\theta}\sqrt{\log\log T}}+O\bigg(T\frac{(\log T)^{\theta}}{(\log\log T)^{3/4}}\bigg).
\]
\end{cor}

It is known that by Littlewood's lemma (see (9.9.1) of \cite{T}) and a sharp estimation of the integral
\begin{equation}\label{int:1}
 \int_T^{2T} \log | \zeta( \sigma_T + it ) - a | dt 
 \end{equation}
would imply the above theorem. We will need two major steps to obtain this. Following Borchsenius and Jessen's framework \cite{BJ}, we introduce the randomized Rieamnn zeta function 
\[
\zeta(\sigma,X)\coloneqq\prod_{p}\left(1-\frac{X(p)}{p^{\sigma}}\right)^{-1},
\]
where $X(p)$ is complex-valued independent and identically distributed random variables on the unit circle $|z|=1$ assigned for each prime $p$, which converges almost surely for $\sigma > 1/2$. In the first step, we prove the discrepancy 
\[
D_{\sigma}(T)\coloneqq\sup_{{\cal R}}\left|\frac{1}{T}\meas\left\{ t\in[T,2T]\,:\,\log\zeta(\sigma+it)\in{\cal R}\right\} -\mathbf{P}\left[ \log\zeta(\sigma,X)\in{\cal R}\right] \right|
\]
is small for $ \sigma = \sigma_T$, where the supremum is taken over all rectangular regions in the complex plane with their sides parallel to real or imaginary axis. Bounding the discrepancy has been studied by various authors (e.g.,
\cite{HM}, \cite{LLR}). In particular,  Lamzouri et al. \cite[Theorem 1.1]{LLR}  showed that  
\begin{equation}
D_{\sigma}(T)\ll_{\sigma}\frac{1}{(\log T)^{\sigma}} \label{eq:LLR.disc}
\end{equation}
 holds for fixed $\sigma>1/2$. Here, the implied constant grows
as $\sigma$ approaches $1/2$.
By the proofs in \cite{LLR} with tracking the dependency of the error terms on $|\sigma-1/2|$, one can deduce the following theorem, which extends Theorem 1.1 in \cite{LLR}.
\begin{thm}
\label{thm:disc}Let $0<\theta<1/2$ and $0<\eta<(1-\theta)/4$ be
fixed, and $T\geq3$ be sufficiently large in terms of $\theta$ and
$\eta$. Then for $1/2+(\log T)^{-\theta}\leq\sigma\leq3/4$, we have
\[
D_{\sigma}(T)\ll_{\eta}\frac{1}{(\log T)^{\eta}}.
\]
\end{thm}

The second step is to find an asymptotic expansion of the joint characteristic function of the real and the imaginary parts of $\log \zeta(\sigma, X)$. We have generalized the estimates in \cite[section 4]{Ts} about modified $J$-Bessel functions and proved the following theorem.
\begin{thm}\label{thm:asymp}
Let $ \theta>0$ be fixed and $\sigma_T =1/2+1/(\log T)^{\theta}$. Let $a<b$ and $c<d$ be real numbers. Define 
\[
\kappa(\sigma_T, X )=\frac{\log\zeta(\sigma_T, X )}{\sqrt{ \pi \psi_T  } },
\]
where 
\[
\psi_T  :=\sum_{p,k}\frac{1}{k^{2}}p^{-2k\sigma_T }=\theta\log\log T+O(1).
\]
Then there exist polynomials 
$g_k ( x,y) $ for $0\leq k \leq 5$ such that $g_0 (x,y) = 1$ and $\deg g_k \leq k$, and for $ T \geq 10$
\begin{align*}
 \mathbf{P} & \left[  \kappa(\sigma_T ,X)\in[a,b] \times [c,d]\right] \\ 
&=  \sum_{ 0 \leq k  \leq 5  } \frac{ 1  }{ \sqrt{\psi_T}^{k}}   \int_c^d \int_a^b g_k ( x,y )  e^{- \pi(x^{2}+y^{2}) } dx dy + O \bigg(  \frac{1}{ ( \log \log T)^3}\bigg).
\end{align*}
\end{thm}
The above theorem and the discrepancy imply the following corollary.
\begin{cor}\label{cor:central limit}
Let $0<\theta<1/2 $ be fixed and $\sigma_T =1/2+1/(\log T)^{\theta}$.
Let $a<b$ and $c<d$ be real numbers. Define 
\[
\kappa(\sigma_T+it )=\frac{\log\zeta(\sigma_T+it )}{\sqrt{ \pi \psi_T } },  \]
then for $T \geq 10$
\begin{align*}
\frac1T \meas& \{ t \in [T,2T] :  \kappa(\sigma_T +it )\in[a,b] \times [c,d] \}  \\ 
&=  \sum_{ 0 \leq k  \leq 5  } \frac{ 1  }{ \sqrt{\psi_T}^{k}}   \int_c^d \int_a^b g_k ( x,y )  e^{- \pi(x^{2}+y^{2}) } dx dy + O \bigg(  \frac{1}{ ( \log \log T)^3}\bigg).
\end{align*}

 \end{cor}
 Note that the first term ($k=0$) in the above sum is nothing but 
 $$    \int_c^d \int_a^b   e^{- \pi(x^{2}+y^{2}) } dx dy .$$
 It is interesting to compare this with Selberg's central limit theorem \cite[Theorem 2]{Sel}, which says that the distribution of $\log \zeta(1/2+it) $, $T\leq t\leq2T$ is Gaussian with mean 0 and variance $\approx \log\log T$. More precisely, let 
\[
\tilde{\kappa}(\sigma,t)=\frac{\log\zeta(\sigma+it)}{\sqrt{ \pi \sum_{p<t}p^{-2\sigma} } }.
\]
For $1/2\leq\sigma\leq1/2+1/(\log T)^{\delta}$ and $a<b$, he proved that  
\begin{align*}
\frac{1}{T}&\meas\left\{ t\in[T,2T]\,:\,a<\Re\tilde{\kappa}(\sigma,t)\leq b\right\} =\int_{a}^{b}e^{-\pi u^{2}}du+O\left(\frac{(\log\log\log T)^2 }{\sqrt{\log\log T}}\right), \\  
\frac{1}{T}& \meas\left\{ t\in[T,2T]\,:\,a<\Im\tilde{\kappa}(\sigma,t)\leq b\right\} =\int_{a}^{b}e^{-\pi u^{2}}du+O\left(\frac{ \log\log\log T  }{\sqrt{\log\log T}}\right).
\end{align*}
So Corollary \ref{cor:central limit} could suggest an asymptotic expansion of Selberg's theorem conjecturally. 

The remaining part of this paper consists of the proofs. Theorem \ref{thm:aval} and Corollary \ref{cor:aval} are proved in Section \ref{section:aval}, Theorem \ref{thm:disc} is proved in Section \ref{section:disc} and Theorem \ref{thm:asymp} is proved in Section \ref{section:asymp}.



\section{The discrepancy bounds near the critical line}\label{section:disc}
 
 In this section we prove Theorem \ref{thm:disc}. Define two functions
 \begin{align*}
\Phi_T  ( {\cal B} )&:= \frac{1}{T}\meas\left\{ t\in[T,2T]\,:\,\log\zeta(\sigma+it)\in{\cal B}\right\},\\
\Phi_{\textrm{rand}}( { \cal B}) & := \mathbf{P}\left[ \log\zeta(\sigma,X)\in{\cal B}\right]
 \end{align*}
 for a Borel set $\cal B$ in $\mathbf{C}$, then we see that 
 \[
D_{\sigma}(T)\coloneqq\sup_{{\cal R}}\left|\Phi_T  ( {\cal R} )  -\Phi_{\textrm{rand}}( { \cal R}) \right|.
\]
We first prove that the Fourier transforms
\[
\widehat{\Phi}_T (u,v)\coloneqq\frac{1}{T}\int_{T}^{2T}\exp\left(2\pi i ( u\Re\zeta(\sigma+it)+ v\Im\zeta(\sigma+it)) \right)dt
\]
and
\[
\widehat{\Phi}_{\textrm{rand}}(u,v)\coloneqq\mathbf{E}\left[ \exp\left(2\pi i ( u\Re\zeta(\sigma,X)+ v\Im\zeta(\sigma,X) ) \right)\right]
\]
 are close, expecting that this would imply $D_{\sigma}(T)$ is small. For a fixed  $1/2<\sigma<1$, Lamzouri et al. \cite{LLR} showed that for any $A\geq1$ there exists $b=b(\sigma,A)$ such that
\[
\widehat{\Phi}(u,v)=\widehat{\Phi}_{\textrm{rand}}(u,v)+O\left(\frac{1}{(\log T)^{A}}\right)
\]
holds for $|u|$, $|v|\leq b(\log T)^{\sigma}$. 
When $\sigma$ approaches $1/2$, we have a similar estimate for a shorter range as follows.
\begin{thm}
\label{thm:phi.hat}Let $0<\theta<1/2$ and $0<\theta_{L}<(1-\theta)/4$
be fixed constants and let $ \sigma_T =1/2+(\log T)^{-\theta}\leq \sigma \leq 3/4$.
Then there exists a $\delta>0$ such that for sufficiently large $T$,
and for all $|u|$, $|v|\leq L := (\log T)^{\theta_{L}}$, we have
\[
\widehat{\Phi}_T(u,v)=\widehat{\Phi}_{\textrm{rand}}(u,v)+O\left( e^{ -(\log T)^{\delta}} \right)  .
\]
\end{thm}
We postpone the proof of Theorem \ref{thm:phi.hat}  to Section \ref{sec:phi.hat}. 
Here, the valid range of $u$ and $v$, or equivalently the size of $L$, is important because the wider valid range implies the smaller discrepancy bound. This is proved in the following proposition.   
\begin{prop}
Let $0<\theta<1/2$ and $0<\theta_{L}<(1-\theta)/4$
be fixed constants and let $ \sigma_T =1/2+(\log T)^{-\theta}\leq \sigma \leq 3/4$. Define
$$ \mathcal{R}_T := \left\{ z\in\mathbf{C}\,:\,\left|\Re z\right|,\,\left|\Im z\right|\leq \Cr{prop.rect}\log\log T\right\} $$
with a large fixed constant $\Cl{prop.rect}$. Then for any rectangle $\mathcal{R}\subseteq  \mathcal{R}_T  $ whose sides are parallel to the axes, we have
\[
\Phi_T (\mathcal{R})=\Phi_{\textrm{rand}}(\mathcal{R})+O\left(\frac{1}{(\log T)^{\theta_{L}}}\right)
\]
as $ T \to \infty$. 
\end{prop}
\begin{proof}[Sketched proof]
We may follow the flow of Section 6 of \cite{LLR}, except we need
to apply \prettyref{thm:phi.hat} instead of Theorem 2.1 of \cite{LLR}
and \prettyref{lem:weak.large.dev} instead of Lemma 6.3 of \cite{LLR}.
The latter change makes extra $\sqrt{\log\log T}$ in (6.10) of \cite{LLR} and the
corresponding integral for $I_{T}(L,s)$. However, by choosing $ \theta_L$ slightly smaller, we could delete the extra $\sqrt{\log \log T}$.
\end{proof}

The above proposition is almost equivalent to \prettyref{thm:disc}. To conclude the proof of \prettyref{thm:disc}, it is enough to  show that
\begin{equation}\label{eqn large dev small 1}
 \Phi_T (\mathcal{R} ) = \Phi_T ( \mathcal{R} \cap \mathcal{R}_T ) +   O\left(\frac{1}{(\log T)^{\theta_{L}}}\right) 
 \end{equation}
and
\begin{equation}\label{eqn large dev small 2}
 \Phi_{\textrm{rand}} (\mathcal{R} ) = \Phi_{\textrm{rand}} ( \mathcal{R} \cap \mathcal{R}_T ) +   O\left(\frac{1}{(\log T)^{\theta_{L}}}\right) 
 \end{equation}
for any rectangles $\mathcal{R}$ whose sides are parallell to the axes.  By \prettyref{lem:weak.large.dev} we can easily see that
\begin{align*}
| \Phi_{\textrm{rand}} (\mathcal{R} ) - \Phi_{\textrm{rand}} ( \mathcal{R} \cap \mathcal{R}_T )  |& \leq \mathbf{P} [  | \log \zeta(\sigma,X) | \geq \Cr{prop.rect}\log\log T ] \\
& \ll \frac{1}{ (\log T)^{\theta_L} }
\end{align*}
for a sufficiently large $\Cr{prop.rect}$ and this proves  \eqref{eqn large dev small 2}.   Similarly, we have that
\begin{equation}\label{eqn large dev 1}\begin{split}
| \Phi_T (\mathcal{R} ) & - \Phi_T ( \mathcal{R} \cap \mathcal{R}_T )  |  \\
& \leq  \frac{1}{T} \meas \{ t \in [T, 2T] :   | \log \zeta(\sigma+it) | \geq \Cr{prop.rect}\log\log T \}  .
\end{split}\end{equation}
By Lemma \ref{lem:large.sieve} or \eqref{proof 3.1 eqn 2}, there is a $\delta>0$ and a set $\mathcal{E} \subset [T,2T]$ with $ |\mathcal{E} | \ll T e^{ - ( \log T)^\delta} $ such that
$$\log \zeta(\sigma+it) = R_Y ( \sigma+it) + O(  e^{ - ( \log T)^\delta} )$$
holds for $ t \in [T,2T] \setminus \mathcal{E}$. Then
\begin{align*}
 \frac{1}{T} &\meas \{ t \in [T, 2T] :   | \log \zeta(\sigma+it) | \geq \Cr{prop.rect}\log\log T \} \\
 & \leq  \frac{| \mathcal{E} |}{T} + \frac{1}{T} \meas \{ t \in [T, 2T] \setminus \mathcal{E}  :   | R_Y(\sigma+it)   | \geq  \Cr{prop.rect}\log\log T -1 \} .
\end{align*}
By \prettyref{lem:weak.large.dev}  there is a constant $C>0$ such that
\begin{equation*}\begin{split} 
\frac{1}{T}& \meas \left\{t \in [T,2T] :  \left|R_{Y}(\sigma+it)\right|\geq   \Cr{prop.rect}  \log\log T  -1 \right\} \\
 & \leq\exp\left(-\frac{(  \Cr{prop.rect} \log\log T -1 )^{2}}{C\sum_{p\leq Y}p^{-2\sigma}}\right) .
\end{split}\end{equation*}
Given $A>0$, we can choose large $\Cr{prop.rect}$ such that 
$$ \exp\left(-\frac{(  \Cr{prop.rect} \log\log T -1 )^{2}}{C\sum_{p\leq Y}p^{-2\sigma}}\right)  \ll \frac{1}{ (\log T)^A } .$$ Thus, given $A>0$ there exists a constant $\Cr{prop.rect}>0$ such that
\begin{equation} \label{eqn log zeta large dev}
  \frac{1}{T} \meas \{ t \in [T, 2T] :   | \log \zeta(\sigma+it) | \geq \Cr{prop.rect}\log\log T \} \ll \frac{1}{ (\log T)^A }.
\end{equation}  
Hence \eqref{eqn large dev small 1} follows from  \eqref{eqn large dev 1} and \eqref{eqn log zeta large dev}. Therefore, we conclude the proof of  \prettyref{thm:disc}.

\subsection{Proof of Theorem \ref{thm:phi.hat}}\label{sec:phi.hat}

We list required lemmas to prove the theorem. We omit the proof if  its cited proof remains valid without much change. 
\begin{lem}
\label{lem:log.approx}\cite[Lemma 2.2]{GS} Let $\sigma_T \leq  \sigma\leq1$
and let $Y$, $T$ be real numbers such that $3 \leq Y \leq T/2 $. Then
there is a set $\mathcal{E}= \mathcal{E}(\sigma,T)$ of measure $\ll T^{5/4-\sigma/2}Y(\log T)^{5}$
such that  
\[
\log\zeta(\sigma+it)=R_{Y}(\sigma+it)+O\left(Y^{(1/2-\sigma)/2}(\log T)^{3}\right)
\]
for all $t\in[T,2T] \setminus \mathcal{E}$, where
$$ R_{Y}(s)  \coloneqq\sum_{m\leq Y}\frac{\Lambda(m)}{\log m} \frac{1}{m^s}=\sum_{\substack{p,k\\
p^{k}\leq Y
}
}\frac{1}{k p^{ ks}} .$$
\end{lem}

\begin{lem}
\label{lem:large.sieve} \cite[Lemma 3.2]{LLR} Let $Y\geq3$ and $k\geq1$ so that
$Y^{k}\leq T^{1/3}$ and let $a_{p}$ for each prime $p$ be a complex
number with $|a_{p}|\leq1$. Then, 
\[
\int_{T}^{2T}\left|\sum_{p\leq Y} \frac{a_p}{p^{\sigma+it}} \right|^{2k}dt\leq k!\left(\sum_{p\leq Y}  \frac{ |a_{p}|^2}{p^{2\sigma}}  \right)^{k}+O\left(T^{-1/3}\right).
\]
Additionally, for any positive integer $k$ we have
\[
\mathbf{E}\left[ \left|\sum_{p\leq Y} \frac{a_{p}X(p)}{p^\sigma}\right|^{2k}\right]  \leq k!\left(\sum_{p\leq Y}  \frac{ |a_{p}|^{2}}{p^{2\sigma}} \right)^{k}.
\]
\end{lem}
\begin{lem}\cite[Lemma 3.3]{LLR} \label{lem:2k mmt} Let $Y\geq3$ and $k\geq1$ so that
$Y^{k}\leq T^{1/3}$. Then there is an absolute constant $\Cl{2k.moment}>0$
such that for any $\sigma > 1/2$ 
\begin{equation}\label{lemma 2.5 eqn 1}
\frac{1}{T}\int_{T}^{2T}\left|R_{Y}(\sigma+it)\right|^{2k}dt\leq\left(\Cr{2k.moment}k\sum_{p\leq Y}p^{-2\sigma}\right)^{k}.
\end{equation}
Also for any $k\geq1$, we have
\begin{equation}\label{lemma 2.5 eqn 2}
\mathbf{E}\left[  \left|R_{Y}(\sigma,X)\right|^{2k}\right]  \leq\left(\Cr{2k.moment}k\sum_{p\leq Y}p^{-2\sigma}\right)^{k} 
\end{equation}
and
\begin{equation}\label{lemma 2.5 eqn 3}
\mathbf{E}\left[ \left|\log\zeta(\sigma,X)\right|^{2k}\right] \leq\left(\Cr{2k.moment}k\sum_{p}p^{-2\sigma}\right)^{k},
\end{equation}
where
$$ R_{Y}(\sigma,X)   \coloneqq\sum_{\substack{p,k\\   p^{k}\leq Y } }\frac{X(p)^{k}}{k p^{k\sigma}} .$$ 
\end{lem}
\begin{proof} 
We divide $R_{Y}$ into three cases as $k=1$, $k=2$ and $k\geq3$. Then
\begin{align*}
\frac{1}{T}&\int_{T}^{2T}\left|R_{Y}(\sigma+it)\right|^{2k}dt\\
& \leq3^{k}   \left(\frac{1}{T}\int_{T}^{2T}\left|\sum_{p\leq Y}\frac{1}{p^{\sigma+it}}\right|^{2k}dt +\frac{1}{T}\int_{T}^{2T}\left|\sum_{p\leq \sqrt{Y}}\frac{1}{2p^{2\sigma+2it}}\right|^{2k}dt+\left(\log\zeta(3\sigma)\right)^{2k}\right).
\end{align*}
By \prettyref{lem:large.sieve}, we prove \eqref{lemma 2.5 eqn 1}. Similarly we can prove \eqref{lemma 2.5 eqn 2}. Taking $Y\to \infty$ in \eqref{lemma 2.5 eqn 2}, we see that \eqref{lemma 2.5 eqn 3} holds.
\end{proof}
\begin{lem}
\label{lem:weak.large.dev}\cite[Lemma 3.4]{LLR} Let $1/2<\sigma\leq3/4$,
$Y\geq2$ and $A=A(T) \geq1$. Then there is an absolute constant $\Cl{large.dev},\Cl{large.dev.2}>0$
such that if 
\begin{equation}\label{condition A}
 \Cr{large.dev} \left(\sum_{p\leq Y}p^{-2\sigma}\right) \leq A^{2}\leq    \frac{ \log T}{\Cr{large.dev}\log Y}\left(\sum_{p\leq Y}p^{-2\sigma}\right) , 
 \end{equation} 
we have
\[
\frac{1}{T}\meas\left\{t \in [T,2T] :  \left|R_{Y}(\sigma+it)\right|\geq A\right\} \leq\exp\left(-\frac{A^{2}}{\Cr{large.dev.2}\sum_{p\leq Y}p^{-2\sigma}}\right).
\]
We also have
\[
\mathbf{P}\left[ \left|R_{Y}(\sigma,X)\right|\geq A\right] \leq\exp\left(-\frac{A^{2}}{\Cr{large.dev.2}\sum_{p\leq Y}p^{-2\sigma}}\right)
\]
and
\[
\mathbf{P}\left[ \left|\log\zeta(\sigma,X)\right|\geq A\right] \leq\exp\left(-\frac{A^{2}}{\Cr{large.dev.2}\sum_{p}p^{-2\sigma}}\right)
\]
without constraints.
\end{lem}
\begin{proof}
By Lemma \ref{lem:2k mmt} we see that
\begin{align*}
\frac{1}{T}& \meas\left\{t \in [T,2T] :   \left|R_{Y}(\sigma+it)\right|\geq A\right\} \\
 & \leq\frac{1}{A^{2k}}\frac{1}{T}\int_{T}^{2T}\left|R_{Y}(\sigma+it)\right|^{2k}dt  \leq\left(\frac{\Cr{2k.moment}k\sum_{p\leq Y}p^{-2\sigma}}{A^{2}}\right)^{k}\leq e^{-k} 
\end{align*}
holds provided that $Y^{k}\leq T^{1/3}$ and $\Cr{2k.moment}ek\sum_{p\leq Y}p^{-2\sigma}\leq A^{2}$.
These hold if we take
 $$k=\left[ \frac{A^2}{ 10 \Cr{2k.moment}\sum_{p \leq Y}  p^{-2\sigma}} \right]$$
  and assume 
$$A^{2}\leq \frac{10   \Cr{2k.moment}  \log T}{  3\log Y   }\left(\sum_{p\leq Y}p^{-2\sigma}\right).$$  
Thus,
\begin{align*}
\frac{1}{T}  \meas\left\{t \in [T,2T] :   \left|R_{Y}(\sigma+it)\right|\geq A\right\} & \leq \exp\left( - \frac{A^2}{ 10 \Cr{2k.moment}\sum_{p\leq Y} p^{-2\sigma}  } +1  \right) \\
 & \leq \exp\left( - \frac{A^2}{  \Cr{large.dev.2}\sum_{p\leq Y} p^{-2\sigma}  }   \right) 
 \end{align*}
 for some large $ \Cr{large.dev.2} >0$.
\end{proof}
\begin{lem}\cite[Lemma 3.4]{Ts}
\label{lem:RY.powers} 
Let $ 1/2 \leq \sigma \leq 1 $ and $Y \geq 10 $. For any positive integers $k$ and $\ell$ such that $ k + \ell \leq  \frac{ \log T}{ 6 \log Y}$, we have 
\begin{align*}
\frac{1}{T}\int_{T}^{2T}& R_{Y}(\sigma+it)^{k}\overline{R_{Y}(\sigma+it)}^{\ell}dt \\
& =\mathbf{E}\left[  R_{Y}(\sigma,X)^{k}\overline{R_{Y}(\sigma,X)}^{\ell}\right]  +O\left(    \frac{  (C(k+\ell)Y)^{(k+\ell)/2}   }{T} \right).
\end{align*}
\end{lem}
 \begin{prop}\cite[Proposition 2.3]{LLR}
\label{prop:exp.z1RYz2RY}  Let
$1/2<\sigma\leq3/4$, and $L,$ $A,$ $Y$, $T\geq3$ be real numbers.
Let $\mathcal{E}$ be the set of $t\in[T,2T]$ with $\left|R_{Y}(\sigma+it)\right|\geq A$ and let $ \mathcal{E}' = [T,2T] \setminus \mathcal{E}$. Suppose that $A$ satisfies \eqref{condition A}, so that 
$$| \mathcal{E} | \leq T \exp\left(-\frac{A^{2}}{\Cr{large.dev.2}\sum_{p\leq Y}p^{-2\sigma}}\right)   .$$
 Then, for complex numbers $z_{1}$, $z_{2}$ with $|z_{i}|\leq L$
for $i=1$, $2$,
\begin{align*}
\frac{1}{T}&\int_{ \mathcal{E}' }   \exp\left(z_{1}R_{Y}(\sigma+it)+z_{2}\overline{R_{Y}(\sigma+it)}\right)dt\\
 = & \mathbf{E}\left[ \exp\left(z_{1}R_{Y}(\sigma,X)+z_{2}\overline{R_{Y}(\sigma,X)}\right)\right]  +O\left(e^{-N}+e^{-A^{2}/\C\sum_{p\leq Y}p^{-2\sigma}}+\frac{1}{\sqrt{T}}\right)
\end{align*}
under the condition that there exists a suitable parameter $N$ satisfying
\[
  L^2 \sum_{p\leq Y}p^{-2\sigma} \ll N , \quad  \sqrt{N}L\sum_{p\leq Y}p^{-2\sigma} \ll A \ll N, \quad  (L^2 Y)^{N}   \leq T.
\]
\end{prop}
\begin{proof}
For simplicity, we put $P(\sigma+it)=z_{1}R_{Y}(\sigma+it)+z_{2}\overline{R_{Y}(\sigma+it)}$
and $P_{\textrm{rand}}(\sigma, X)=z_{1}R_{Y}(\sigma,X)+z_{2}\overline{R_{Y}(\sigma,X)}$.
For $|P|\leq A$, we have
\[
\exp(P)=\sum_{n<N}\frac{1}{n!}P^{n}+O\left(\frac{A^{N}}{N!}\right).
\]
By taking $N\geq10A$ and using Stirling's formula, the error term
is $O(e^{-N})$. Hence,
$$ \frac{1}{T} \int_{ \mathcal{E}' } \exp(P( \sigma+it)) dt = \sum_{n<N}\frac{1}{n!} \frac{1}{T} \int_{ \mathcal{E}' } P(\sigma+it)^{n} dt +O(T^{-1} e^{-N}) .$$

Next, we want to add the integrals of $P^n$ on the set $\mathcal{E}$. By the Cauchy-Schwarz inequality, Lemmas \ref{lem:2k mmt}--\ref{lem:weak.large.dev}, and the inequality 
$$|P(\sigma+it)|^{2}\leq4L^{2}\left|R_{Y}(\sigma+it)\right|^{2}, $$
  we have 
\begin{align*}
\frac{1}{T} \int_{\mathcal{E}}\left|P(\sigma+it)\right|^{n}dt & \leq\sqrt{\frac{1}{T}\meas(\mathcal{E})}\left(\frac{1}{T} \int_{T}^{2T}\left|P(\sigma+it)\right|^{2n}dt\right)^{1/2}\\
 & \leq\exp\left(-\frac{A^{2}}{2\Cr{large.dev.2}\sum_{p\leq Y}  p^{-2\sigma}}\right)\left(4\Cr{2k.moment}NL^{2}\sum_{p\leq Y}p^{-2\sigma}\right)^{n} 
\end{align*}
 for $n<N$ and for $ A$ satisfying the conditions in Lemma \ref{lem:weak.large.dev}, thus we see that 
 \begin{align*}
 \frac{1}{T}   \int_{ \mathcal{E}' } \exp(P( \sigma+it)) dt 
 = &   \sum_{n<N}\frac{1}{n!} \frac{1}{T} \int_T^{2T} P(\sigma+it)^{n} dt + O \left( \frac{e^{-N}}{T}  \right)  \\
 & +O\left(   \exp\left(-\frac{A^{2}}{2\Cr{large.dev.2}\sum_{p\leq Y}  p^{-2\sigma}} + 4\Cr{2k.moment}NL^{2}\sum_{p\leq Y}p^{-2\sigma}\right)    \right) .
 \end{align*}
If we choose $A$ so that 
\[
A\geq4\sqrt{\Cr{large.dev.2}\Cr{2k.moment}N}L\sum_{p\leq Y}p^{-2\sigma},
\]
we have
\[
\frac{1}{T}\int_{\mathcal{E}' }\exp(P(\sigma+it))dt=\sum_{n<N}\frac{1}{n!}\frac{1}{T}\int_{T}^{2T}P(\sigma+it)^{n}dt+O\left(e^{-N}+e^{-A^{2}/ ( 4\Cr{large.dev.2}\sum_{p \leq Y}  p^{-2\sigma}) }\right).
\]
 By Lemma \ref{lem:RY.powers} we see that
\begin{align*}
\frac{1}{T}\int_{T}^{2T}P(\sigma+it)^{n}dt & =\sum_{k+\ell=n}{n \choose k}z_{1}^{k}z_{2}^{\ell}\frac{1}{T}\int_{T}^{2T}R_{Y}(\sigma+it)^{k}\overline{R_{Y}(\sigma+it)}^{\ell }dt\\
 & =\mathbf{E}\left[ P_{\textrm{rand}}(\sigma,X)^{n}\right] +O\left(\frac{1}{T}(4CnL^2 Y)^{n/2}\right)
\end{align*}
and thus
\begin{align*}
\frac{1}{T}\int_{\mathcal{E}'  }\exp(P(\sigma+it))dt= & \sum_{n<N}\frac{1}{n!}\mathbf{E}\left[ P_{\textrm{rand}}(\sigma,X)^{n}\right] +O\left(e^{-N}+e^{-A^{2}/4\Cr{large.dev.2}\sum p^{-2\sigma}}\right)\\
 & +O\left(\frac{1}{T}\sum_{n<N}\frac{(4CnL^2 Y)^{n/2}}{n!}\right)
\end{align*}
If we assume that $(L^2 Y)^{N}\leq T$, then the last $O$-term is $O(1/\sqrt{T})$.

Finally, we want to add the tail of the above sum. By the inequalty $\left|P_{\textrm{rand}}(\sigma,X)\right|^{2}\leq4L^{2}\left|R_{Y}(\sigma,X)\right|^{2}$ and Lemma \ref{lem:2k mmt}, we see that  
\[
\sum_{n\geq N}\frac{1}{n!}\mathbf{E}\left[ \left|P_{\textrm{rand}}(\sigma,X)\right|^{n}\right] \leq\sum_{n\geq N}\left(\frac{\C  L\sqrt{\sum_{p\leq Y}p^{-2\sigma}}}{\sqrt{n}}\right)^{n}\leq e^{-N}
\]
if $N\gg L^2 \sum_{p\leq Y}p^{-2\sigma}.$ Therefore,
\begin{align*}
\frac{1}{T}\int_{\mathcal{E}'  }\exp(P(\sigma+it))dt= & \mathbf{E}\left[ \exp( P_{\textrm{rand}}(\sigma,X)) \right] +O\left(e^{-N}+e^{-A^{2}/4\Cr{large.dev.2}\sum p^{-2\sigma}}   +  \frac{1}{ \sqrt{T}} \right).
\end{align*}
\end{proof}

Now, we adjust the parameters in the proposition. Let $\sigma_T = 1/2+(\log T)^{-\theta}$, $ \sigma_T \leq \sigma \leq 3/4$, $A=(\log T)^{\theta_{A}}$,
$L= ( \log T)^{ \theta_L}$, $N=(\log T)^{\theta_{N}}$ and $Y=\exp\left((\log T)^{\theta_{1}}\right)$. Then the conditions are met for sufficiently large $T$
 if
\[
  2 \theta_{L}< \theta_N ,\qquad  \theta_{L}+\frac{\theta_{N}}{2}< \theta_A < \theta_N,\qquad\theta_{1}+\theta_{N}<1, \qquad 2 \theta_A < 1- \theta_1 .
\]
They hold if we assume $ \theta_L < (1-\theta_1 )/4$ and let $ \theta_A = \tfrac12 \theta_L + \tfrac38 (1-\theta_1)$ and $ \theta_N = \tfrac12 (1-\theta_1 ) $.  By Proposition \ref{prop:exp.z1RYz2RY} with these choices, we see that there is a constant $0< \delta <1 $ and a set $\mathcal{E} \subset [T,2T]$ with
$$| \mathcal{E} | \leq T e^{ - ( \log T)^\delta }  ,$$ 
such that
\begin{equation}\label{proof 3.1 eqn 1}\begin{split}
\frac{1}{T}\int_{ [T,2T]\setminus \mathcal{E}}  & \exp\left(z_{1}R_{Y}(\sigma+it)+z_{2}\overline{R_{Y}(\sigma+it)}\right)dt \\
 & = \mathbf{E}\left[ \exp\left(z_{1}R_{Y}(\sigma,X)+z_{2}\overline{R_{Y}(\sigma,X)}\right)\right]  +O(  e^{ - ( \log T)^\delta } ) 
\end{split}\end{equation}
holds for any $|z_1|, |z_2 | \leq (\log T)^{\theta_L} $.  Let $\mathcal{E}_{1}$ be the set of $t\in[T,2T]$ for which the approximation
of \prettyref{lem:log.approx} fails. By \prettyref{lem:log.approx} we see that
\[
\log\zeta(\sigma+it)=R_{Y}(\sigma+it)+O\left(\exp\left(-\frac{1}{2}(\log T)^{(\theta_{1}-\theta)/2}\right)\right)
\]
for all $ t \in [T, 2T] \setminus \mathcal{E}_1 $ and 
$$ | \mathcal{E}_1| \ll    T \exp\left(-\frac{1}{3}(\log T)^{ 1-\theta }\right)  $$
provided that $ \theta_1 < 1- \theta  $. Thus, assuming  $ \theta < \theta_1 < 1-\theta $, there is a $0< \delta < 1 $ such that 
\begin{equation}\label{proof 3.1 eqn 2}
\log\zeta(\sigma+it)=R_{Y}(\sigma+it)  +O(  e^{ - ( \log T)^\delta } )  
\end{equation}
for all $ t \in  [T, 2T] \setminus \mathcal{E}_1 $ and 
$$ | \mathcal{E}_1| \ll    T  e^{ - ( \log T)^\delta }. $$
Hence, we only need the assumptions $ \theta_L < (1-\theta_1)/4$ and $ \theta < \theta_1 < 1-\theta $ for the discussion in this paragraph.

Let $ 0 < \theta < 1/2$, $ 0 < \theta_L < ( 1-\theta)/4 $ and $ \theta_1 = \theta + \min \{  ( 1-\theta)/4-\theta_L , (1/2-\theta)/2 \} $. These choices imply the two conditions $ \theta_L < (1-\theta_1)/4$ and $ \theta < \theta_1 < 1-\theta $ above, so that \eqref{proof 3.1 eqn 1} and \eqref{proof 3.1 eqn 2} hold for all $ t \in [T, 2T] \setminus (\mathcal{E} \cup \mathcal{E}_1 )$ with a small exceptional set $ |\mathcal{E} \cup \mathcal{E}_1|  \ll    T  e^{ - ( \log T)^\delta }  $. Thus, by \eqref{proof 3.1 eqn 2} 
\begin{align*}
\widehat{\Phi}_T (u,v)=&  \frac{1}{T}\int_{T}^{2T}   \exp\left(2\pi i\left(u\Re R_{Y}(\sigma+it)+v\Im R_{Y}(\sigma+it)\right)\right)dt  \\
&  +O( L e^{ - ( \log T)^\delta })
\end{align*}
for all $ |u|, |v| \leq L = ( \log T)^{\theta_L} $. By \eqref{proof 3.1 eqn 1} we have 
\begin{align*}
\widehat{\Phi}_T (u,v)=&  \mathbf{E} \left[    \exp\left(2\pi i\left(u\Re R_{Y}(\sigma, X )+v\Im R_{Y}(\sigma,X )\right)\right) \right]   +O( L e^{ - ( \log T)^\delta })
\end{align*}
for all $ |u|, |v| \leq L = ( \log T)^{\theta_L} $. By Lemma 4.1 in \cite{LLR} we finally prove that
\begin{align*}
\widehat{\Phi}_T (u,v)= \widehat{\Phi}_{\textrm{rand}}(u,v)  +O(   e^{ - ( \log T)^\delta })
\end{align*}
holds for all $ |u|, |v| \leq L = ( \log T)^{\theta_L} $ with a smaller choice of $\delta >0$. This concludes the proof of Theorem \ref{thm:phi.hat}.

\section{Estimates on the randomized zeta function}\label{section:asymp}

 We first present an asymptotic expansion of the density function for $\log \zeta(\sigma,X)$.  \begin{prop}
\label{prop:density.randomized}Let $\sigma$ be a real number with
$1/2<\sigma\leq3/4$. 
\begin{enumerate}
\item There is a smooth density function $F_{\sigma}(x,y)$ such that for
any region $\mathcal{B}$
\[
\mathbf{P}\left[ \log\zeta(\sigma,X)\in\mathcal{B}\right] =\iint_{\mathcal{B}}F_{\sigma}(x,y)dxdy.
\]
\item $F_{\sigma}(x,y)$ has an asymptotic series expansion
\[
F_\sigma (x,y)= 
\sum_{\substack{m,n\geq0\\ m+n\leq5}
}\frac{c_{m,n}(1/ \sqrt{\psi(\sigma) }) }{\psi(\sigma)^{m+n+1}}x^{m}y^{n}e^{-(x^{2}+y^{2})/\psi(\sigma)}+O\left(\frac{1}{\psi^{4}(\sigma)}\right),
\]
where $c_{m,n}(\alpha)$ is a polynomial in $\alpha$ with $ c_{0,0}(0) = \pi^{-1}$  and
\[
\psi(\sigma)\coloneqq\sum_{p,k}\frac{1}{k^{2}p^{2k\sigma}}=\log\frac{1}{\sigma-1/2}+O(1).
\]
\end{enumerate}
\end{prop}

\prettyref{thm:asymp} is an easy consequence of the above proposition. We shall prove \prettyref{thm:asymp} here.
\begin{proof}[Proof of \prettyref{thm:asymp}]
By Proposition  \ref{prop:density.randomized} we see that 
\begin{align*}
 \mathbf{P}  \left[ \kappa(\sigma_T ,X)\in[a,b] \times [c,d]\right]
&= \int_{c \sqrt{ \pi \psi_T }}^{d  \sqrt{ \pi \psi_T }} \int_{a \sqrt{ \pi \psi_T }}^{b \sqrt{ \pi \psi_T }} F_{\sigma_T} ( x,y) dx dy \\
&=\pi \psi_T \int_c^d \int_a^b   F_{\sigma_T} (x \sqrt{ \pi \psi_T },y\sqrt{ \pi \psi_T }) dx dy.
\end{align*}
Since
 \[
F_{\sigma_T} (x,y)=\sum_{\substack{m,n\geq0\\
m+n\leq5
}
}\frac{c_{m,n}(1/ \sqrt{\psi_T }) }{\psi_T^{m+n+1}}x^{m}y^{n}e^{-(x^{2}+y^{2})/\psi_T}+O\left(\frac{1}{\psi_T^{4} }\right) ,
\]
we find that
\begin{align*}
 \mathbf{P} & \left[ \kappa(\sigma_T ,X)\in[a,b] \times [c,d]\right] \\ 
&=    \sum_{\substack{m,n\geq0\\
m+n\leq5
}
}\frac{c_{m,n}(1/ \sqrt{\psi_T }) \sqrt{\pi}^{m+n+2} }{ \sqrt{\psi_T}^{m+n }} \int_c^d \int_a^b  x^{m}y^{n}e^{- \pi(x^{2}+y^{2}) } dx dy + O \bigg(  \frac{1}{ \psi_T^3}\bigg).
\end{align*}
Write $ c_{m,n} (\alpha) = \sum_{\ell=0}^\infty c_{m,n,\ell} \alpha^\ell$, then
\begin{align*}
 \mathbf{P} & \left[ \kappa(\sigma_T ,X)\in[a,b] \times [c,d]\right] \\ 
&=  \sum_{\ell \geq 0 }   \sum_{\substack{m,n\geq0\\
m+n\leq5
}
}\frac{c_{m,n,\ell } \sqrt{\pi}^{m+n+2} }{ \sqrt{\psi_T}^{m+n +\ell }} \int_c^d \int_a^b  x^{m}y^{n}e^{- \pi(x^{2}+y^{2}) } dx dy + O \bigg(  \frac{1}{ \psi_T^3}\bigg) \\
&=  \sum_{k  \geq 0 } \frac{ 1  }{ \sqrt{\psi_T}^{k}}   \int_c^d \int_a^b  \sum_{m+n \leq k }   c_{m,n,k-m-n } \sqrt{\pi}^{m+n+2}   x^{m}y^{n}e^{- \pi(x^{2}+y^{2}) } dx dy + O \bigg(  \frac{1}{ \psi_T^3}\bigg).
\end{align*}
Define $ g_k ( x,y ) =  \sum_{m+n \leq k }   c_{m,n,k-m-n } \sqrt{\pi}^{m+n+2}   x^{m}y^{n}$, then 
\begin{align*}
 \mathbf{P} & \left[ \kappa(\sigma_T ,X)\in[a,b] \times [c,d]\right] \\ 
&=  \sum_{k  \geq 0 } \frac{ 1  }{ \sqrt{\psi_T}^{k}}   \int_c^d \int_a^b g_k ( x,y )  e^{- \pi(x^{2}+y^{2}) } dx dy + O \bigg(  \frac{1}{ \psi_T^3}\bigg).
\end{align*}
In particular, 
$$ g_0 (x,y) = c_{0,0,0} \pi = c_{0,0}(0) \pi = 1 . $$
\end{proof}

The remaining part of the section is devoted to proving Proposition \ref{prop:density.randomized}. We first want
to understand the Fourier transform of the cumulative distribution
function, or equivalently, certain complex moment of $\zeta(\sigma,X)$.
Recall that 
\[
\hat{\Phi}_{\textrm{rand}}(u,v)=\mathbf{E}\left[\exp\left(2\pi iu\Re\log\zeta(\sigma,X)+2\pi iv\Im\log\zeta(\sigma,X)\right)\right]
\]
for real numbers $u,v$. It can be rewritten as 
\[
\hat{\Phi}_{\textrm{rand}}(u,v)=\mathbf{E}\left[\zeta(\sigma,X)^{\pi i(u-iv)}\zeta(\sigma,\overline{X})^{\pi i(u+iv)}\right].
\] 
Since $X(p)$ are assumed to be independent, we have
\[
\hat{\Phi}_{\textrm{rand}}(u,v)=\prod_{p}J(\pi u,\pi v,p^{-\sigma}),
\]
 where
\begin{equation}\label{def juvw}
J(u,v,w)   =\mathbf{E}\left[\exp(- 2iu\Re \log(1-wX) - 2iv\Im \log (1-wX)\right] 
\end{equation} 
with $z=u+iv$ and $0<w<1$. Since $- \log (1-wX) \approx wX $ for $ w \approx 0 $, we can see that $J(u,v,w)\approx J_{0}(w\left|z\right|)$, where $J_{0}$ is the classical
$J$-Bessel function. So we expect that $J(u,v,w)$ and $J_0 (w|z|)$ share similar properties. 
\begin{lem}
\label{lem:J.properties} Let $u$, $v$, $w$ are real variables
and $0<w<1$, then the followings are true.
\begin{enumerate}
\item $|J(u,v,w)|\leq1$ for all $u$, $v$, $w$.
\item $|J(u,v,w)|\ll1/(w\sqrt{u^{2}+v^{2}})^{1/2}$ for $w\sqrt{u^{2}+v^{2}}\geq1$.
\end{enumerate}
\end{lem}
\begin{proof}
The first inequality follows by \eqref{def juvw}. The
second is an application of van der Corput method; one can find the
proof in Theorem 12 of \cite{JW}.
\end{proof}

\begin{lem}
\label{lem:coeff.estimate}Let $u$, $v$, $w$ be real numbers and
$0<w<1$. We put $z=u+iv$. Then we have a series expansion
\[
J(u,v,w)=\sum_{k,l\geq0}\frac{i^{k+l}}{k!l!}a_{k,l}(w)z^{k}\bar{z}^{l}
\]
for any $z$. Let $0<r<1$, then there is a constant $\delta_{r}>0$
such that for $|w|\leq r$ and $|z|\leq\delta_{r}$, we have the series
expansion
\[
\log J(u,v,w)=\sum_{k,l\geq1}\frac{i^{k+l}}{k!l!}b_{k,l}(w)z^{k}\bar{z}^{l}.
\]
Moreover, the followings are true.
\begin{enumerate}
\item  $a_{0,0}(w)=1$, $a_{k,0}(w)=a_{0,k}(w)=0$ for $k \geq 1$, 
and $a_{k,l}(w)=a_{l,k}(w) > 0 $ for $k,l \geq  1$.
\item $b_{k,l}(w)$ is real and 
\[
b_{1,1}(w)=a_{1,1}(w)=\sum_{m=1}^\infty \frac{1}{m^{2}}w^{2m}.
\]
\item $a_{k,l}(w)$, $b_{k,l}(w)\ll_{k,l}w^{2\max(k,l)}$ for $k,l \geq 1$..
\item There is a constant $C_{r}$ so that for $0<w\leq r$, 
\end{enumerate}
\begin{align*}
a_{k,l} & \leq C_{r}^{k+l}w^{k+l}\\
\left|b_{k,l}\right| & \leq C_{r}^{k+l}\min(k,l)^{k+l}w^{k+l}
\end{align*}
\end{lem}
\begin{proof}
We find that a series expansion of $J(u,v,w)$ in terms of $z=u+iv$ and $ \bar{z}$ is
\begin{align*}
J(u,v,w) & =\mathbf{E}\left[ \exp ( -iz \log (1-wX ) - i \overline{z} \log (1-w\overline{X} ) )  \right]\\
  & =\sum_{k,l\geq0}\frac{i^{k+l}}{k!l!}a_{k,l}(w)z^{k}\bar{z}^{l},
\end{align*}
where
\[
a_{k,l}(w)\coloneqq\mathbf{E}\left[ \left(-\log(1-wX)\right)^{k}\left(-\log(1-w\overline{X})\right)^{l}\right] .
\]
It is easy to see that $a_{0,0}(w)=1$, $a_{k,0}(w)=a_{0,l}(w)=0$ for
$k$, $l\geq1$ and $a_{1,1}(w)=\sum_{m\geq1}w^{2m}/m^{2}$.
By the symmetry to change $X$ to $\overline{X}$, we have that $a_{k,l}(w)=a_{l.k}(w)$.
We next plug in a power series expansion of $- \log ( 1- w X)$ to see that 
\begin{align*}
a_{k,l}(w) & =\mathbf{E}\left[ \sum_{n_{1},\ldots,n_{k}\geq1}\frac{w^{n_{1}+\cdots+n_{k}}}{n_{1}\cdots n_{k}}X^{n_{1}+\cdots+n_{k}} \sum_{m_{1},\ldots,m_{l}\geq1}\frac{w^{m_{1}+\cdots+m_{l}}}{m_{1}\cdots m_{l}}\overline{X}^{m_{1}+\cdots+m_{l}}\right] \\
 & =\sum_{t\geq\max(k,l)}\left(\sum_{\substack{\sum n_{i}=t\\
n_{i}\geq1
}
}\frac{1}{n_{1}\cdots n_{k}}\right)\left(\sum_{\substack{\sum m_{j}=t\\
m_{j}\geq1
}
}\frac{1}{m_{1}\cdots m_{l}}\right)w^{2t}.
\end{align*}
Thus $a_{k,l}(w) $ is real and positive for $k$, $l\geq1$ and
$a_{k,l}(w)\ll_{k,l}w^{2\max(k,l)}$. By the maximum modulus
principle on $-\log(1-w)/w$, we find that 
  for $0<w\leq r <1 $ 
\begin{equation}
a_{k,l}(w)\leq w^{k+l}\left(\frac{-\log\left(1-r\right)}{r}\right)^{k+l}.\label{eq:a.kl.estimate}
\end{equation}
By \eqref{eq:a.kl.estimate}, 
\begin{align*}
\left|J(u,v,w)-1\right| & \leq\sum_{k,l\geq1}\frac{a_{k,l}(w)}{k!l!}|z|^{k+l} \leq\left(\exp(C_{r}|z|w)-1\right)^{2} <1
\end{align*}
 for $|w|\leq r$ and $|z|\leq \delta_{r}:=(2rC_{r})^{-1} $.  Then we can find a series
expansion as
\begin{align*}
\log J(u,v,w) & =\sum_{n\geq1}\frac{(-1)^{n-1}}{n}\left(J(u,v,w)-1\right)^{n}\\
&   =\sum_{n\geq1}\frac{(-1)^{n-1}}{n}\left(     \sum_{k,l\geq0}\frac{i^{k+l}}{k!l!}a_{k,l}(w)z^{k}\bar{z}^{l} -1 \right)^n.
\end{align*}
Put $ \log J(u,v,w) =  \sum_{k,l\geq1}\frac{i^{k+l}}{k!l!}b_{k,l}(w)z^{k}\bar{z}^{l} $ and compare the coefficients, then we find that 
\[
b_{k,l}(w)=\sum_{n\leq\min(k,l)}\frac{(-1)^{n-1}}{n}\sum_{\substack{k_{1}+\cdots+k_{n}=k\\
l_{1}+\cdots+l_{n}=l\\
k_{i},l_{i}\geq1
}
}{k \choose k_{1},\ldots,k_{n}}{l \choose l_{1},\ldots,l_{n}}a_{k_{1},l_{1}}(w)\cdots a_{k_{n},l_{n}}(w).
\]
In particular, we have $b_{1,1}(w)=a_{1,1}(w)$. Since $\sum_{i}\max(k_{i},l_{i})\geq\max(k,l)$,
we have $\left|b_{k,l}(w)\right|\ll_{k,l}w^{2\max(k,l)}$. We can deduce that
\begin{align*}
\left|b_{k,l}(w) \right| & \leq\left(-\frac{\log(1-r)}{r}\right)^{k+l}w^{k+l}\sum_{n\leq\min(k,l)}n^{k+l-1}.\\
 & \leq\left(-\frac{\log(1-r)}{r}\right)^{k+l}w^{k+l}\min(k,l)^{k+l}
\end{align*}
  for $(k,l)\neq(1,1)$ and   $0<w\leq r<1$.
\end{proof}
 
We next find an asymptotic expansion of $\widehat{\Phi}_{\textrm{rand}}(u,v)$.
\begin{lem}
\label{lem:Phi.hat.expansion}There exists an absolute constant $\delta>0$
such that for $u^{2}+v^{2}\leq\delta$ and $\sigma>1/2$ 
\begin{equation}\begin{split}
\widehat{\Phi}_{\textrm{rand}}(u,v)=& e^{-\pi^{2}(u^{2}+v^{2})\psi(\sigma)} \\
& \times \bigg(1+\sum_{\substack{k,l\geq1\\
3\leq k+l\leq5
}
} \tilde{a}_{k,l} (\sigma)(u+iv)^{k}(u-iv)^{l}+O\left((u^{2}+v^{2})^{3}\right)\bigg).\label{eq:Phi.hat.expression}
\end{split}\end{equation}
Here, $ \tilde{a}_{k,l} (\sigma)$ for $k,l\geq1$ and $3\leq k+l\leq5$  is a smooth function on $\sigma\geq1/2$ given by 
\[
 \tilde{a}_{k,l} (\sigma)=\frac{(\pi i)^{k+l}}{k!l!}\sum_{p}b_{k,l}(p^{-\sigma}).
\] 
\end{lem}
\begin{proof}
By \prettyref{lem:coeff.estimate} with $r=1/\sqrt{2}$, there is
$\delta_{0}>0$ such that
\begin{align*}
\widehat{\Phi}_{\textrm{rand}}(u,v) & =\prod_{p}J(\pi u,\pi v,p^{-\sigma})\\
 & =\exp\bigg(-\pi^{2}(u^{2}+v^{2})\sum_{p}b_{1,1}(p^{-\sigma})+\sum_{\substack{k,l\geq1\\
(k,l)\neq(1,1)
}
}\frac{(\pi i)^{k+l}}{k!l!}z^{k}\bar{z}^{l}\sum_{p}b_{k,l}(p^{-\sigma})\bigg)
\end{align*}
holds for $z=u+iv$ with $|z|\leq\delta_{0}$. Note that $\sum_{p}b_{1,1}(p^{-\sigma})=\psi(\sigma)$. By Lemma \ref{lem:coeff.estimate} the terms in the exponent with $k+l\geq6$
contributes
\begin{align*}
\left|\sum_{k+l\geq6}\frac{(-i)^{k+l}}{k!l!}z^{k}\bar{z}^{l}\sum_{p}b_{k,l}(p^{-\sigma})\right| & \ll\sum_{k+l\geq6}\frac{\min(k,l)^{k+l}}{k!l!}C^{k+l}|z|^{k+l}\\
 & \ll|z|^{6}\sum_{k+l\geq6}\left(Ce^{2}|z|\right)^{k+l-6}\\
 & \ll (u^{2}+v^{2})^3
\end{align*}
assuming $|z|\leq ( 2Ce^{2})^{-1}$. Similarly, terms with $3\leq k+l\leq5$ contributes
\begin{align*}
\exp & \left(   \sum_{3\leq k+l\leq5}\frac{(\pi i)^{k+l}}{k!l!}z^{k}\bar{z}^{l}\sum_{p}b_{k,l}(p^{-\sigma})\right) \\
&  =1+\sum_{3\leq k+l\leq5}\frac{(\pi i)^{k+l}}{k!l!}z^{k}\bar{z}^{l}\sum_{p}b_{k,l}(p^{-\sigma})+O\left(|z|^{6}\right).
\end{align*}
Thus we have for $|z|\leq\delta\coloneqq\min(\delta_{0},(2Ce^{2})^{-1})$
\[
\widehat{\Phi}_{\textrm{rand}}(u,v)= e^{ -\pi^{2}(u^{2}+v^{2})\psi(\sigma)}  \left(1+\sum_{3\leq k+l\leq5} \tilde{a}_{k,l} (\sigma)z^{k}\bar{z}^{l}+O\left(|z|^{6}\right)\right),
\]
where
\[
 \tilde{a}_{k,l} (\sigma)\coloneqq\frac{(\pi i)^{k+l}}{k!l!}\sum_{p}b_{k,l}(p^{-\sigma}).
\]
\end{proof}
Next we prove that $\widehat{\Phi}_{\textrm{rand}}(u,v)$ decays rapidly.
\begin{lem}
\label{lem:Phi.hat.decay} There are  constants $C>0$ such that
\[
\left|\widehat{\Phi}_{\textrm{rand}}(u,v)\right|\ll\exp\left(-C \psi(\sigma)(u^{2}+v^{2})\right)
\]
holds for $0\leq u^{2}+v^{2}\leq\exp\left(\sqrt{\log\psi(\sigma)}\right)$ and 
\[
\left|\widehat{\Phi}_{\textrm{rand}}(u,v)\right|\ll\exp\left(-C  \frac{ \left(u^{2}+v^{2}\right)^{1/(2\sigma)}}{ \log(u^{2}+v^{2}) } \right)
\]
for $u^{2}+v^{2}\geq\exp\left(\sqrt{\log\psi(\sigma)}\right)$.
\end{lem}
\begin{proof}
  \prettyref{lem:Phi.hat.expansion} implies that the above inequality holds for  $u^{2}+v^{2}\leq\delta$. Hence, we assume that $u^{2}+v^{2}\geq\delta$.
  
Let $\Pi_{1}$ be the product of $|J(\pi u,\pi v,p^{-\sigma})|$ for
$p^{2\sigma}\leq\pi^{2}(u^{2}+v^{2})/4$ and $\Pi_{2}$ be for $p^{2\sigma}>(u^{2}+v^{2})/4$.
 By \prettyref{lem:coeff.estimate} we see that
\begin{align*}
J(\pi u,\pi v,p^{-\sigma}) & =1-\pi^{2}(u^{2}+v^{2})p^{-2\sigma}+O\left((u^{2}+v^{2})^{2}p^{-4\sigma}\right)\\
 & =\exp\left(-\pi^{2}(u^{2}+v^{2})p^{-2\sigma}+O\left((u^{2}+v^{2})^{2}p^{-4\sigma}\right)\right) 
\end{align*}
holds for $p^{2\sigma}>\pi^{2}(u^{2}+v^{2})/4$.
Then
\[
(u^{2}+v^{2})\sum_{p^{2\sigma} >\pi^{2}(u^{2}+v^{2})/4}p^{-4\sigma} =o(u^{2}+v^{2})
\]
as $u^2 +v^2 \to \infty$. If $u^{2}+v^{2}\leq\exp\left(\sqrt{\psi(\sigma)}\right)$, then by the prime number theorem we have 
 \[
\sum_{p>\left(\pi^{2}(u^{2}+v^{2})/4\right)^{1/2\sigma}}p^{-2\sigma}\gg\psi(\sigma).
\]
Thus
\begin{align*}
\Pi_{2} & \leq\exp\left(-C \psi(\sigma)(u^{2}+v^{2})\right)
\end{align*}
for some constant $C >0$. Since $\Pi_{1}\leq1$ by \prettyref{lem:J.properties},
we have the first  bound of $\widehat{\Phi}_{\textrm{rand}}(u,v)    $ for $u^{2}+v^{2}\leq\exp\left(\sqrt{\psi(\sigma)}\right)$.

For $u^{2}+v^{2}\geq\exp\left(\sqrt{\psi(\sigma)}\right)$ we use
property 2 of \prettyref{lem:J.properties} to have 
$$J(\pi u,\pi v,p^{-\sigma})\leq1/e$$
for $p^{-2\sigma}(u^{2}+v^{2})\geq C ^{-1}$. Then
we have
\begin{equation}
\Pi_{1}\leq\exp\left(-\sum_{p\leq(C (u^{2}+v^{2}))^{1/(2\sigma)}}1\right)\leq\exp\left(-C \frac{(u^{2}+v^{2})^{1/(2\sigma)}}{\log(u^{2}+v^{2})}\right)\label{eq:uv.large}.
\end{equation}
This implies the second bound.
\end{proof}

Now we prove Proposition \ref{prop:density.randomized}. 
  In \prettyref{lem:Phi.hat.decay}, we showed that the Fourier
transform of the measure ${\Phi}_{\textrm{rand}}$ decays rapidly and thus ${\Phi}_{\textrm{rand}}$ has the smooth density function $F_\sigma (x,y)$
(see \cite{BJ}) defined by
\[
F_\sigma (x,y)=\iint_{\mathbf{R}^{2}}\widehat{\Phi}_{\textrm{rand}}(u,v)e^{-2\pi i(ux+vy)}dxdy
\]
that satisfies
\[
{\Phi}_{\textrm{rand}}( \mathcal{B})   =\iint_{\mathcal{B}}F_\sigma (x,y)dxdy.
\]
  Let $\delta>0$ be the fixed constant to satisfy \prettyref{lem:Phi.hat.expansion}
and we put the polynomial part of \eqref{eq:Phi.hat.expression}
\[
P(\sigma: u,v) :=1+\sum_{\substack{k,l\geq1\\
3\leq k+l\leq5
}
} \tilde{a}_{k,l} (\sigma)(u+iv)^{k}(u-iv)^{l}.
\]
  Using estimates
in \prettyref{lem:Phi.hat.decay}, we have
\[
\iint_{u^{2}+v^{2}\geq\delta}\left|\widehat{\Phi}_{\textrm{rand}}(u,v)\right|dudv\ll e^{-C\psi(\sigma)},
\]
for some constant $C>0$. For
$u^{2}+v^{2}\geq\delta$, we also note that for $k$, $l\leq6$
\[
\iint_{u^{2}+v^{2}\geq\delta}u^{k}v^{l}e^{-\psi(\sigma)(u^{2}+v^{2})} dudv \ll\exp(-C\psi(\sigma)).
\]
Since $ \tilde{a}_{k,l} (\sigma) =  \tilde{a}_{k,l} (1/2) + O( \sigma-1/2)$ by Lemma \ref{lem:Phi.hat.expansion}, we have
\begin{align*}
F_\sigma (x,y) &  =\iint_{\mathbf{R}^{2}}P(\sigma: u,v)e^{-\pi^{2}\psi(\sigma)(u^{2}+v^{2})-2\pi i(ux+vy)}dudv\\
 & \qquad+O\left(\iint_{\mathbf{R}^{2}}(u^{2}+v^{2})^{3}e^{-\pi^{2}\psi(\sigma)(u^{2}+v^{2})}dudv+e^{-C\psi(\sigma)}\right)\\
  & =\iint_{\mathbf{R}^{2}}P( u,v)e^{-\pi^{2}\psi(\sigma)(u^{2}+v^{2})-2\pi i(ux+vy)}dudv +O \bigg(  \frac{  (\sigma-1/2)}{ \psi(\sigma)^{5/2}} + \frac{1}{ \psi(\sigma)^4 }  \bigg)  ,
\end{align*}
where
$$ P( u,v) := P( 1/2 : u,v)= 1+\sum_{\substack{k,l\geq1\\
3\leq k+l\leq5
}
} \tilde{a}_{k,l} (1/2)(u+iv)^{k}(u-iv)^{l}. $$
 We apply the change of variables $\tilde{u}=\sqrt{\psi(\sigma)}\left(u+ix/\pi\psi(\sigma)\right)$
and $\tilde{v}=\sqrt{\psi(\sigma)}\left(v+iy/\pi\psi(\sigma)\right)$
and use the classical contour shift argument. Then the main integral
becomes
\begin{align*}
 & \iint_{\mathbf{R}^{2}}P(u,v)e^{-\pi^{2}\psi(\sigma)(u^{2}+v^{2})-2\pi i(ux+vy)}dudv\\
= & \frac{e^{-(x^{2}+y^{2})/\psi(\sigma)}}{\psi(\sigma)}\iint_{\mathbf{R}^{2}}P\left(\frac{\tilde{u}}{\sqrt{\psi(\sigma)}}-\frac{ix}{\pi\psi(\sigma)},\frac{\tilde{v}}{\sqrt{\psi(\sigma)}}-\frac{iy}{\pi\psi(\sigma)}\right)e^{-\pi^{2}(\tilde{u}^{2}+\tilde{v}^{2})}d\tilde{u}d\tilde{v}.
\end{align*}
We want to split the variables $\tilde{u}$, $\tilde{v}$ to $ x$, $y$ and  write 
\[
P\left(\frac{\tilde{u}}{\sqrt{\psi(\sigma)}}-\frac{ix}{\pi\psi(\sigma)},\frac{\tilde{v}}{\sqrt{\psi(\sigma)}}-\frac{iy}{\pi\psi(\sigma)}\right)=\sum_{\substack{m,n\\
m+n\leq5
}
}\frac{x^{m}y^{n} }{\psi(\sigma)^{m+n}}P^{(m,n)}\left(\frac{\tilde{u}}{\sqrt{\psi(\sigma)}},\frac{\tilde{v}}{\sqrt{\psi(\sigma)}}\right),
\]
where 
\[
P^{(m,n)}(u,v)=\frac{1}{m!n!(\pi i)^{m+n}}\frac{\partial^{m+n}}{\partial u^{m}\partial v^{n}}P(u,v)
\]
  is a polynomial. Then 
\begin{align*}
 & \iint_{\mathbf{R}^{2}}P(u,v)e^{-\pi^{2}(u^{2}+v^{2})\psi(\sigma)}e^{-2\pi i(ux+vy)}dudv\\
= & \sum_{m,n}\frac{x^{m}y^{n}}{\psi(\sigma)^{m+n+1}}e^{-(x^{2}+y^{2})/\psi(\sigma)}\iint_{\mathbf{R}^{2}}P^{(m,n)}\left(\frac{\tilde{u}}{\sqrt{\psi(\sigma)}},\frac{\tilde{v}}{\sqrt{\psi(\sigma)}}\right)e^{-\pi^{2}(\tilde{u}^{2}+\tilde{v}^{2})}d\tilde{u}d\tilde{v}\\
 = &  \sum_{m.n}c_{m,n}( 1/\sqrt{ \psi(\sigma)})\frac{x^{m}y^{n}}{\psi(\sigma)^{m+n+1}}e^{-(x^{2}+y^{2})/\psi(\sigma)},
\end{align*}
where
\begin{align*}
c_{m,n}(\alpha ) : = & \iint_{\mathbf{R}^{2}}P^{(m,n)}(\alpha u , \alpha v)e^{-\pi^{2}(u^{2}+v^{2})}dudv 
\end{align*} 
 is a polynomial. In particular,
 $$ c_{0,0} (0) =\frac{ P^{(0,0)} (0,0)}{\pi} =\frac{ P  (0,0)}{\pi} = \frac{1}{ \pi} .$$
  This completes the proof of Proposition \ref{prop:density.randomized}.


\section{The $a$-values of $\zeta(s)$ near the $1/2$-line}\label{section:aval}

As mentioned in the introduction, we want to estimate the integral 
\begin{equation}\label{int log zeta a}
\frac{1}{T}\int_{T}^{2T}\log|\zeta(\sigma_{T}+it)-a|dt.
\end{equation}
Since $\zeta(s)$ has infinitely many $a$-values, the above integral have a lot of logarithmic singularities. One way to get around this difficulty is using high moments of $ \log |\zeta(\sigma+it) -a | $. This can be easily derived from the proof of Proposition 2.5 in \cite{LLR} by tracking the dependency on $| \sigma- 1/2 |$, so we write the following lemma without its proof.
\begin{lem}\label{lem:L.2k.logZ-a}
Let $ a$ be a fixed complex number. There exists an absolute constant $C>0$ such that for any real number $k \geq 1  $ we have
\begin{align*}
  \frac1T \int_T^{2T} | \log | \zeta(\sigma_T + it) - a  | |^{2k}  dt   \ll   & (Ck)^{4k}(\log T)^{3k\theta} + ( Ck \log \log T)^{2k}(\log T)^{6k\theta} \\
  &  + ( C\log \log T)^{4k} ( \log T)^{6k\theta}.  
  \end{align*}
\end{lem}

Now we estimate \eqref{int log zeta a}. Let $A$ be a large fixed constant, $ L= ( \log T)^\eta $ with fixed $0<\eta < (1- \theta )/4 $ and let
\[
{\cal R}_{1}=[\log|a|+1/L,\,A\log\log T]\times[-A\log\log T,A\log\log T]
\]
\[
{\cal R}_{2}=[-A\log\log T,\,\log|a|-1/L]\times[-A\log\log T,A\log\log T]
\]
\[
{\cal R}_{3}=(\log|a|-1/L,\,\log|a|+1/L)\times[-A\log\log T,A\log\log T]
\]
and 
\[
{\cal S}_{j}=\{t\in[T,2T]:\log\zeta(\sigma_{T}+it)\in{\cal R}_{j}\}
\]
for $j=1,2,3$. Let ${\cal S}=[T,2T]\setminus\bigcup_{j}{\cal S}_{j}$. By \eqref{eqn log zeta large dev}  we see that  
$$\frac1T |\mathcal{S}|  \leq \frac1T \meas \{t\in[T,2T]: | \log\zeta(\sigma_{T}+it) | \geq A \log \log T \} \ll \frac{1}{ ( \log T)^{10}} $$
holds for a sufficiently large constant $ A>0$. By the Cauchy-Schwarz inequality and Lemma \ref{lem:L.2k.logZ-a}
with $k=1$, 
\begin{align*}
\bigg|\frac{1}{T}\int_{{\cal S}}\log\left|\zeta(\sigma_{T}+it)-a\right|dt\bigg| & \leq\bigg(\frac{1}{T}|{\cal S}|\bigg)^{1/2}\bigg(\frac{1}{T}\int_{T}^{2T}(\log\left|\zeta(\sigma_{T}+it)-a\right|)^{2}dt\bigg)^{1/2}\\
 & \ll\frac{(\log\log T)^2}{(\log T)^{5-3\theta}} .
\end{align*}

We next consider the integral on $\mathcal{S}_3 $, which contains the logarithmic singularities of the integrand.
By Theorem \ref{thm:disc}, we have  
\begin{align*}
\frac{1}{T}|{\cal S}_{3}| & =\Phi_{T}({\cal R}_{3})=\Phi_{\textrm{rand}}({\cal R}_{3})+O\bigg(\frac{1}{(\log T)^\eta }\bigg).
\end{align*}
By Proposition \ref{prop:density.randomized} we see that
$$ \Phi_{\textrm{rand}}({\cal R}_{3}) = \iint_{ {\cal R}_3 } F_{ \sigma_T} ( x,y ) dx dy  \ll \frac{1}{(\log T)^\eta } . $$
Thus, 
$$ \frac{1}{T}|{\cal S}_{3}|  \ll  \frac{1}{(\log T)^\eta } .$$
By H\"older's inequality and Lemma \ref{lem:L.2k.logZ-a} with $k=\log\log T$,
\begin{align*}
\bigg|\frac{1}{T}\int_{{\cal S}_{3}}\log\left|\zeta(\sigma_{T}+it)-a\right|dt\bigg| & \leq\bigg(\frac{1}{T}|{\cal S}_{3}|\bigg)^{1- \tfrac{1}{2k}}\bigg(\frac{1}{T}\int_{T}^{2T}(\log\left|\zeta(\sigma_{T}+it)-a\right|)^{2k}dt\bigg)^{\tfrac{1}{2k} }\\
 & \ll    \frac{ ( \log\log T)^2 }{ (\log T)^{\eta- 3\theta}}.
\end{align*}
By Theorem \ref{thm:disc} and the same method as in pp. 486\textendash 489 of \cite{Le3},
we see that (cf. Theorem 2.4 and Section 8.1 of \cite{LLR})
\begin{align*}
\frac{1}{T}\int_{{\cal S}_{j}}\log\left|\zeta(\sigma_{T}+it)-a\right|dt & =\int_{{\cal R}_{j}}\log\left|e^{u+iv}-a\right|d\Phi_{T}(u,v)\\
 & =\int_{{\cal R}_{j}}\log\left|e^{u+iv}-a\right|d\Phi_{\textrm{rand}}(u,v)+O\bigg(\frac{(\log\log T)^{2}}{( \log T)^\eta}\bigg)
\end{align*}
for $j=1,2$. We next apply Proposition \ref{prop:density.randomized} and deduce that 
\begin{align*}
 & \int_{{\cal R}_{j}}\log\left|e^{x+iy}-a\right|d\Phi_{\textrm{rand}}(x,y)\\
 & =\int_{{\cal R}_{j}}\log\left|e^{x+iy}-a\right|F_{\sigma_T} (x,y)dxdy\\
 & =\int_{{\cal R}_{j}}\log\left|e^{x+iy}-a\right|\sum_{\substack{m,n\geq0\\ m+n\leq5}
}\frac{c_{m,n}(1/ \sqrt{\psi_T  }) }{\psi_T^{m+n+1}}x^{m}y^{n}e^{-(x^{2}+y^{2})/\psi_T} dxdy+O\bigg(\frac{1}{\log\log T }\bigg)
\end{align*}
for $j=1,2$, where $\psi_T = \theta \log \log T + O(1)$ is defined in Theorem \ref{thm:asymp}.  
Combining all these estimates we have 
\begin{align*}
\frac{1}{T}& \int_{T}^{2T}\log\left|\zeta(\sigma_{T}+it)-a\right|dt \\
= & \sum_{j=1,2}  \sum_{\substack{m,n\geq0\\ m+n\leq5}
}\frac{c_{m,n}(1/ \sqrt{\psi_T  }) }{\psi_T^{m+n+1}}  \int_{{\cal R}_{j}}\log\left|e^{x+iy}-a\right|  x^{m}y^{n}e^{-(x^{2}+y^{2})/\psi_T} dxdy  +O\bigg(   \frac{1}{\log\log T  }  \bigg)
\end{align*}
provided that $  \eta > 3 \theta $. Such $\eta $ exists only when $ 3\theta < ( 1-\theta)/4$. Thus we need to assume that $ \theta < 1/13$. 

Next we  estimate the integral on ${\cal R}_{1}$. Since
$$\log\left|e^{x+iy}-a\right| = x+ \log\left|1-ae^{-x-iy}\right| = x - \Re \sum_{k=1}^{\infty}\frac{a^{k}}{k} e^{-kx-iky}$$
for $ x+iy \in {\cal R}_1 $, we see that 
\begin{equation}\begin{split}\label{eqn R1}
 & \int_{{\cal R}_{1}}\log\left|e^{x+iy}-a\right|x^{m}y^{n}e^{-(x^{2}+y^{2})/\psi_T }dxdy\\
 & =\int_{{\cal R}_{1}}x^{m+1}y^{n}e^{-(x^{2}+y^{2})/\psi_T }dxdy-\Re\bigg[\sum_{k=1}^{\infty}\frac{a^{k}}{k}\int_{{\cal R}_{1}}e^{-kx-iky}x^{m}y^{n}e^{-(x^{2}+y^{2})/\psi_T }dxdy\bigg].
\end{split}\end{equation}
The main term of \eqref{eqn R1} is 
\begin{align*}
\int_{{\cal R}_{1}} & x^{m+1}y^{n}e^{-(x^{2}+y^{2})/\psi_T }dxdy\\
 & =\int_{\log|a|+1/L}^{A\log\log T}x^{m+1}e^{-x^{2}/\psi_T }dx\int_{-A\log\log T}^{A\log\log T}y^{n}e^{-y^{2}/\psi_T }dy\\
 & =(\sqrt{\psi_T })^{m+n+3}\int_{(\log|a|+1/L)/\sqrt{\psi_T }}^{A\log\log T/\sqrt{\psi_T}}x^{m+1}e^{-x^{2}}dx\int_{-A\log\log T/\sqrt{\psi_T }}^{A\log\log T/\sqrt{\psi_T }}y^{n}e^{-y^{2}}dy\\
 & =(\sqrt{\psi_T })^{m+n+3}\int_{(\log|a|)/\sqrt{\psi_T }}^{\infty}x^{m+1}e^{-x^{2}}dx\int_{-\infty}^{\infty}y^{n}e^{-y^{2}}dy+O\bigg(\frac{1}{(\log T)^{c}}\bigg)
\end{align*}
for some $c>0$ and for $ m,n \leq 5 $. The error terms of \eqref{eqn R1} are 
\begin{align*}
\sum_{k=1}^{\infty} & \frac{a^{k}}{k}\int_{{\cal R}_{1}}e^{-kx-iky}x^{m}y^{n}e^{-(x^{2}+y^{2})/\psi_T}dxdy\\
 & =\sum_{k=1}^{\infty}\frac{a^{k}}{k}\int_{\log|a|+1/L}^{A\log\log T}x^{m}e^{-kx-x^{2}/\psi_T}dx\int_{-A\log\log T}^{A\log\log T}y^{n}e^{-iky-y^{2}/\psi_T}dy.
\end{align*}
It is easy to see that
$$\frac{a^{k}}{k}\int_{\log|a|+1/L}^{A\log\log T}x^{m}e^{-kx-x^{2}/\psi_T}dx \ll \frac{ e^{-k/L}}{k} (\sqrt{\psi_T})^{m+1} .$$
The $y$-integral above is 
\begin{align*}
\int_{-A\log\log T}^{A\log\log T}y^{n}e^{-iky-y^{2}/\psi_T}dy & =\int_{-\infty}^{\infty}y^{n}e^{-iky-y^{2}/\psi_T}dy+O\bigg(\frac{1}{(\log T)^{c}}\bigg)\\
 & =(\sqrt{\psi_T})^{n+1}\int_{-\infty}^{\infty}y^{n}e^{-ik\sqrt{\psi_T}y-y^{2}}dy+O\bigg(\frac{1}{(\log T)^{c}}\bigg)\\
 & =(\sqrt{\psi_T})^{n+1}e^{-\frac{k^{2}}{4}\psi_T}\int_{-\infty}^{\infty}y^{n}e^{-(y+ik\sqrt{\psi_T}/2)^{2}}dy+O\bigg(\frac{1}{(\log T)^{c}}\bigg)\\
 & =(\sqrt{\psi_T})^{n+1}e^{-\frac{k^{2}}{4}\psi_T}\int_{-\infty}^{\infty}(y-ik\sqrt{\psi_T }/2)^{n}e^{-y^{2}}dy+O\bigg(\frac{1}{(\log T)^{c}}\bigg)\\
 & \ll k^{n}(\sqrt{\psi_T})^{2n+1}e^{-\frac{k^{2}}{4}\psi_T }+\frac{1}{(\log T)^{c}}\\
 & \ll\frac{1}{(\log T)^{c}}
\end{align*}
for some $c>0$ and all $k\geq1$, $0\leq n\leq6$. Hence, 
\begin{align*}
\sum_{k=1}^{\infty}\frac{a^{k}}{k}\int_{{\cal R}_{1}}e^{-kx-iky}x^{m}y^{n}e^{-(x^{2}+y^{2})/\psi_T}dxdy & \ll\frac{1}{(\log T)^{c}}\sum_{k=1}^{\infty}\frac{e^{-k/L}}{k}(\sqrt{\psi_T})^{m+1}\\
 & \ll\frac{(\log\log T)^{(m+3)/2}}{(\log T)^{c}}.
\end{align*}
Thus, we have 
\begin{align*}
 & \int_{{\cal R}_{1}}\log\left|e^{x+iy}-a\right|x^{m}y^{n}e^{-(x^{2}+y^{2})/\psi_T}dxdy\\
 & = (\sqrt{\psi_T})^{m+n+3}\int_{(\log|a|)/\sqrt{\psi_T}}^{\infty}x^{m+1}e^{-x^{2}}dx\int_{-\infty}^{\infty}y^{n}e^{-y^{2}}dy+O\bigg(\frac{1}{(\log T)^{c}}\bigg)
\end{align*}
for some $c>0$. Similarly, we can estimate the integral on ${\cal R}_{2}$ and obtain that 
\begin{align*}
 & \int_{{\cal R}_{2}}\log\left|e^{x+iy}-a\right|x^{m}y^{n}e^{-(x^{2}+y^{2})/\psi_T}dxdy\\
 & =\log|a| (\sqrt{\psi_T})^{m+n+2}\int_{-\infty}^{(\log|a|)/\sqrt{\psi_T}} x^{m}e^{-x^{2}}dx\int_{-\infty}^{\infty}y^{n}e^{-y^{2}}dy+O\bigg(\frac{1}{(\log T)^{c}}\bigg).
\end{align*}
Notice that the $y$-integral $\int_{-\infty}^{\infty}y^{n}e^{-y^{2}}dy$ is vanishing when $n $ is odd and that
\[
2 \int_{-\infty}^{\infty}y^{2}e^{-y^{2}}dy= \int_{-\infty}^{\infty}e^{-y^{2}}dy= \sqrt{\pi} .
\] 
We also note that $ c_{0,0}(\alpha )=\pi^{-1} +O( |\alpha|^3 ) $ and $ c_{m,n }(\alpha)=O( | \alpha|^{3-m-n} ) $ for $ m+n = 1, 2 $  hold as $\alpha \to 0 $. Therefore, 
\begin{align*}
\frac{1}{T}& \int_{T}^{2T}\log\left|\zeta(\sigma_{T}+it)-a\right|dt \\
= &   \sum_{\substack{m,n\geq0\\ m+n \leq 2  \\ n~~ even}
}\frac{c_{m,n}(1/ \sqrt{\psi_T  }) }{ \sqrt{\psi_T}^{m+n-1}} \int_{(\log|a|)/\sqrt{\psi_T}}^{\infty}x^{m+1}e^{-x^{2}}dx\int_{-\infty}^{\infty}y^{n}e^{-y^{2}}dy   \\
 +  &  \log|a|  \sum_{\substack{m,n\geq0\\ m+n\leq 1 \\ n~~ even}
}\frac{c_{m,n}(1/ \sqrt{\psi_T  }) }{ \sqrt{\psi_T}^{m+n}}  \int_{-\infty}^{(\log|a|)/\sqrt{\psi_T}} x^{m}e^{-x^{2}}dx\int_{-\infty}^{\infty}y^{n}e^{-y^{2}}dy    +O\bigg(   \frac{1}{\log\log T  }  \bigg) \\
= & \frac{1}{\sqrt{\pi}}  \sqrt{\psi_T}   \int_{(\log|a|)/\sqrt{\psi_T}}^{\infty}x e^{-x^{2}}dx+ \frac{1}{\sqrt{\pi}}   \log|a|   \int_{-\infty}^{(\log|a|)/\sqrt{\psi_T}}  e^{-x^{2}}dx     +O\bigg(   \frac{1}{\log\log T  }  \bigg) 
\end{align*}
Since 
\[
\int_{(\log|a|)/\sqrt{\psi_T}}^{\infty}xe^{-x^{2}}dx=\frac{1}{2}e^{-(\log|a|)^{2}/\psi_T}=\frac{1}{2}-\frac{(\log|a|)^{2}}{2\psi_T}+O\bigg(\frac{1}{(\psi_T)^{2}}\bigg)
\]
and
\begin{align*}
  \int_{-\infty}^{(\log|a|)/\sqrt{\psi_T}} e^{-x^{2}}dx & =\int_{-\infty}^0 e^{-x^{2}}dx+\int_0^{(\log|a|)/\sqrt{\psi_T}} 1+O(x^{2})dx\\
 & =\frac{\sqrt{\pi}}{2}+ \frac{\log|a|}{\sqrt{\psi_T}}+O\bigg(\frac{1}{\psi_T^{3/2}}\bigg),
\end{align*}
we have  
\begin{align*}
\frac{1}{T}& \int_{T}^{2T}\log\left|\zeta(\sigma_{T}+it)-a\right|dt \\
 & = \frac{\sqrt{\psi_T}}{2\sqrt{\pi}}+ \frac{1}{2}\log|a|+\frac{1}{2\sqrt{\pi}}\frac{(\log|a|)^{2}}{\sqrt{\psi_T}}+O\bigg(\frac{1}{ \log \log T }\bigg).
\end{align*}

We can now estimate $N_a ( \sigma_T; T, 2T)$ and ready to prove Theorem \ref{thm:aval}. By Littlewood's lemma, we see that 
\begin{equation}
\begin{split}\int_{\sigma_{T}(\theta_{1})}^{\sigma_{T}(\theta_{2})} & N_{a}(w;T,2T)dw\\
= & \frac{1}{2\pi}\int_{T}^{2T}\log\left|\zeta(\sigma_{T}(\theta_{1})+it)-a\right|dt\\
 & -\frac{1}{2\pi}\int_{T}^{2T}\log\left|\zeta(\sigma_{T}(\theta_{2})+it)-a\right|dt+O(\log T)\\
= & \frac{T}{4\pi}\bigg(\frac{1}{\sqrt{\pi}}+\frac{(\log|a|)^{2}}{\sqrt{\psi(\sigma_{T}(\theta_{1}))\psi(\sigma_{T}(\theta_{2}))}}\bigg)(\sqrt{\psi(\sigma_{T}(\theta_{1}))}-\sqrt{\psi(\sigma_{T}(\theta_{2}))})\\
 & +O\bigg(  \frac{T }{\log\log T }  \bigg).
\end{split}
\label{eqn:littlewood}
\end{equation}
Put $\theta_{1}=\theta$ and $\theta_{2}=\theta-h$ with $h>0$ in
\eqref{eqn:littlewood}, then 
\begin{align*}
\int_{\sigma_{T}(\theta)}^{\sigma_{T}(\theta-h)}N_{a}(w;T,2T)dw & \leq(\sigma_{T}(\theta-h)-\sigma_{T}(\theta))N_{a}(\sigma_{T}(\theta);T,2T)\\
 & =\frac{(\log T)^{h}-1}{(\log T)^{\theta}}N_{a}(\sigma_{T}(\theta);T,2T).
\end{align*}
Thus, 
\begin{align*}
N_{a}(\sigma_{T}(\theta);T,2T)\geq & \frac{T}{4\pi}\bigg(\frac{1}{\sqrt{\pi}}+\frac{(\log|a|)^{2}}{\sqrt{\psi(\sigma_{T}(\theta))\psi(\sigma_{T}(\theta-h))}}\bigg)\frac{\sqrt{\psi(\sigma_{T}(\theta))}-\sqrt{\psi(\sigma_{T}(\theta-h))}}{\sigma_{T}(\theta-h)-\sigma_{T}(\theta)}\\
 & +O\bigg(T\frac{(\log T)^{\theta}}{\log\log T}\frac{1}{(\log T)^{h}-1}\bigg).
\end{align*}
Let $h=(\log\log T)^{-5/4}$. By Taylor's theorem we have 
\[
\frac{\psi(\sigma_{T}(\theta))-\psi(\sigma_{T}(\theta-h))}{\sigma_{T}(\theta-h)-\sigma_{T}(\theta)}=(\log T)^{\theta}\bigg(1+O\bigg(\frac{1}{(\log\log T)^{1/4}}\bigg)\bigg)
\]
and 
\[
\psi(\sigma_{T}(\theta))-\psi(\sigma_{T}(\theta-h))=O\bigg(\frac{1}{(\log\log T)^{1/4}}\bigg).
\]
By the prime number theorem 
\[
\psi(\sigma_{T}(\theta))=\theta\log\log T+O(1).
\]
Therefore, 
\begin{align*}
N_{a}(\sigma_{T}(\theta);T,2T)\geq\frac{T(\log T)^{\theta}}{8\pi^{3/2}\sqrt{\theta}\sqrt{\log\log T}}+O\bigg(T\frac{(\log T)^{\theta}}{(\log\log T)^{3/4}}\bigg).
\end{align*}

To find an upper bound, we put $\theta_{1}=\theta+h$ and $\theta_{2}=\theta$
with the same $h=(\log\log T)^{-5/4}$ in \eqref{eqn:littlewood},
then 
\begin{align*}
\int_{\sigma_{T}(\theta+h)}^{\sigma_{T}(\theta)}N_{a}(w;T,2T)dw & \geq(\sigma_{T}(\theta)-\sigma_{T}(\theta+h))N_{a}(\sigma_{T}(\theta);T,2T).
\end{align*}
Repeating the above process, we obtain 
\begin{align*}
N_{a}(\sigma_{T}(\theta);T,2T)\leq\frac{T(\log T)^{\theta}}{8\pi^{3/2}\sqrt{\theta}\sqrt{\log\log T}}+O\bigg(T\frac{(\log T)^{\theta}}{(\log\log T)^{3/4}}\bigg).
\end{align*}
Therefore, 
\begin{align*}
N_{a}(\sigma_{T}(\theta);T,2T)=\frac{T(\log T)^{\theta}}{8\pi^{3/2}\sqrt{\theta}\sqrt{\log\log T}}+O\bigg(T\frac{(\log T)^{\theta}}{(\log\log T)^{3/4}}\bigg).
\end{align*}

\noindent  School of Mathematics, Korea Institute for Advanced Study\\
85 Hoegiro Dongdaemun-gu, Seoul 02455, Korea\\
\textsl{E-mail address} : Junsoo Ha (junsooha@kias.re.kr)

\noindent Department of Mathematics, Incheon National University, \\
119 Academy-ro, Yeonsu-gu, Incheon, 22012, Korea\\
\textsl{E-mail address} : Yoonbok Lee (leeyb@inu.ac.kr, leeyb131@gmail.com)

\end{document}